\documentclass[12pt]{umj_eng3} 
\usepackage[english]{babel}
\usepackage{amsmath}
\usepackage{amssymb}
\usepackage{amsfonts, dsfont, graphicx, srcltx}

\newtheorem{lemma}{Lemma}[section]
\newtheorem{applemma}{Lemma A\!\!}[section]
\newtheorem{theorem}{Theorem}[section]

\firstpage{1}

\renewcommand{\leq}{\leqslant}
\renewcommand{\geq}{\geqslant}

\newtheorem{definition}{Definition}
\newtheorem{remark}{Remark}

\def\mR{{\mathbb R}}

\def\p{{\partial}} 

\def\ud{\,\mathrm{d}}

\def\e{\varepsilon}

\def\l{\left}
\def\r{\right}
\def\ud{\,\mathrm{d}}

\def\GRAD{\nabla}

\usepackage{xspace}
\newcommand{\ie}{i.e.,\@\xspace}
\newcommand{\eg}{e.g.,\@\xspace}

\newcommand{\calD}{{\mathcal D}}
\newcommand{\calB}{{\mathcal B}}
\newcommand{\calC}{{\mathcal C}}

\newcommand{\calO}{{\mathcal O}}

\newcommand{\bfjlg}{\boldsymbol}         

\newcommand{\bx}{{\bfjlg x}}

\newcommand{\by}{{\bfjlg y}}
\newcommand{\be}{{\bfjlg e}}
\newcommand{\bnu}{{\bfjlg \nu}}
\newcommand{\bxi}{{\bfjlg \xi}}

\begin{document}
\thispagestyle{empty}

\title[Almost periodic solutions of Poisson's equation]
{Regularity of almost periodic solutions of Poisson's equation} 

\author{\`{E}. Muhamadiev, M. Nazarov}

\address{\`{E}rgash Muhamadiev
\newline\hphantom{iii} Department of Information Systems
  and Technology
\newline\hphantom{iii} Vologda State University 
\newline\hphantom{iii} Vologda, Russia}
\email{emuhamadiev@rambler.ru}

\address{Murtazo Nazarov
\newline\hphantom{iii} Division of Scientific Computing
\newline\hphantom{iii} Department of Information Technology
\newline\hphantom{iii} Uppsala University
\newline\hphantom{iii} Uppsala, Sweden}
\email{murtazo.nazarov@it.uu.se}

\thanks{\sc \`{E}. Muhamadiev, M. Nazarov
  Regularity of almost periodic solutions of Poisson's equation}
\thanks{\copyright \ 2020 \`{E}. Muhamadiev, M. Nazarov}
\thanks{\rm This material is based upon work supported by Esseen scholarship at Uppsala University.}

\maketitle 
{
\small
\begin{quote}

  \noindent{\bf Abstract. }
  This paper discusses some regularity of almost periodic solutions of the Poisson's equation $-\Delta u = f$ in $\mR^n$, where $f$ is an almost periodic function. It has been proved by Sibuya [{\em Almost periodic solutions of Poisson’s equation.} Proc. Amer. Math. Soc., 28:195--198, 1971.] that if $u$ is a bounded continuous function and solves the Poisson's equation in the distribution sense, then $u$ is an almost periodic function. In this work, we relax the assumption of the usual boundedness into boundedness in the sense of distribution which we refer to as {\em a bounded generalized function}. The set of bounded generalized functions are wider than the set of usual bounded functions. Then, upon assuming that $u$ is a bounded generalized function and solves the Poisson's equation in the distribution sense, we prove that this solution is bounded in the usual sense, continuous and almost periodic. Moreover, we show that the first partial derivatives of the solution $\p u/ \p x_i$, $i=1, \ldots, n$, are also continuous, bounded and almost periodic functions. The technique is based on extending a representation formula using Green's function for Poisson's equation for solutions in the distribution sense. Some useful properties of distributions are also shown that can be used to study other elliptic problems.
  
  \medskip

  
  
\end{quote}
}

\section{Introduction}
We are interested in solving the following Poisson's equation 
\begin{equation}\label{eq:poisson}
-\Delta u = f,
\end{equation}
where $u: \mR^n \rightarrow \mR$ is the solution and $f: \mR^n \rightarrow \mR$ is given almost periodic source function
and $\Delta = \text{div} \cdot \GRAD$ is the Laplace operator. A function $f$ is called an almost periodic of $\bx$ if $f$ is
continuous in $\mR^n$, and for every sequence of points $\{\bx_n\} \in \mR^n$, the corresponding sequence $\{ f(\bx + \bx_n) \}$ contains a uniformly convergent sub-sequence. Our interest is to study the behavior of the solution, which is obtained when the source function is almost periodic. 

The theory of almost periodic solutions of ordinary differential equations was started by early work of Bohr \& Neugebauer \cite{Bohr_1926} and Jean Favard \cite{Favard_1928, Favard_1933}.  Jean Favard proved the following theorem: if all homogeneous limit equations have no non-zero bounded solutions, and the original system has a bounded solution, then this solution is almost periodic. This result raised the problem of the existence of a bounded solution of a system with almost periodic coefficients.  Later, it was shown Muhamadiev \cite{Muhamadiev_1972}, that the property of limit systems mentioned in the theorem of Favard, guarantees the existence of a bounded solution of the original system, and consequently its almost periodicity. 

The question of the behavior of solutions of Poisson's equation with almost periodic source functions was first addressed by Sibuya \cite{Sibuya_1971}. Later, Sell in \cite{Sell_1973} extended the result of Sibuya and Favard to linear systems of partial differential equations with almost periodic coefficients and source functions. We remark here that, the extension of the theory of almost periodic solution to more general classes of differential equations such pseudo-differential operators were studied by Shubin \cite{Shubin_1978}.

The following theorem was proven by Sibuya \cite{Sibuya_1971}:
\medskip
\begin{theorem}[Sibuya]\label{th:sibuya}
  Let $f(\bx)$ be an almost periodic function of $\bx$ in $\mR^n$, and let $u(\bx)$ be a bounded continuous function of $\bx$ in $\mR^n$. Assume that $u(\bx)$ is a solution of \eqref{eq:poisson} in the sense of distribution. Then $u(\bx)$ is almost periodic with respect to   $\bx$ in $\mR^n$.
\end{theorem}
\medskip

Theorem~\ref{th:sibuya} shows that if a bounded continuous function $u$ solves equation \eqref{eq:poisson} in the distribution 
sense, then it is almost periodic. The main goal of this paper is to address the following key questions: $(i)$ to investigate the possibility of relaxing the assumption of Theorem~\ref{th:sibuya}, \ie
consider a wider class of solutions rather than bounded continuous functions; $(ii)$ to investigate properties of the partial derivatives of such solutions, \ie boundedness, continuity, and almost periodicity.

We have to stress that in this paper we assume that the solution $u$ of the Poisson's equation is bounded in the distribution sense, whereas in \cite{Sibuya_1971} and \cite{Sell_1973} the solution is assumed to be bounded in the strong sense.

This paper is organized as follows. Some definitions including the definition of the bounded generalized functions are presented in Section~\ref{Sec:Preliminaries}.  Then we present our main results in Section~\ref{Sec:main_results}.  The proof of the main theorems is presented in Section~\ref{sec:proof:th}. Some technical lemmas are proved in detail in Section~\ref{Sec:tec_lemma} and in the appendix.

\subsection{Preliminaries}\label{Sec:Preliminaries} Throughout this paper we follow the notations and definitions consistent with \cite{Evans_2010}. Let $\Omega$ be an open subset of $\mR^n$, we denote the closure of it by $\bar\Omega$.  We often use an
open ball of radius $r>0$ centered at the point $\bx$ by $\calB^0(\bx, r) = \{\by \in \mR^n: |\bx-\by| < r\}$, and we denote by
$\calB(\bx, r) = \{\by \in \mR^n: |\bx-\by| \le r\}$ a closed ball.  For a given function $g(\bx) \equiv g(x_1, \ldots, x_n)$, $\bx\in \Omega$, we denote the normal derivative by $\frac{\p g}{\p \bnu} (\bx) = \bnu \cdot D g(\bx)$, where $\bnu$ is an
outward pointing unit normal to the boundary $\p\Omega$, and $Dg := (\frac{\p g}{ \p x_1}, \ldots, \frac{\p g}{ \p x_n})$ is the
gradient of $g$.

Furthermore, we use the so-called Green's functions to write the representation formula for the Poisson's equations, see \eg 
\cite{Evans_2010}. 
\medskip
\begin{definition}
Green's function for an open set $\Omega = \calB^0(\bx, 1)$ in $\mR^n$ is
\begin{equation}\label{eq:G(x,y)}
  G(\bx,\by) := \Phi(\by - \bx) - \Phi(|\bx| (\by - \tilde\bx)),
  \quad \big( (\bx, \by) \in \Omega, \bx \not= \by \big),
\end{equation}
where 
$
\tilde \bx = \frac{\bx}{|\bx|^2}
$
and
\begin{equation}\label{eq:Phi(x,y)}
  \Phi(\bx) = 
  \begin{cases}
      -\frac{1}{2n} \log(\bx), \quad &n = 2,\\
      \frac{1}{n(n-2)\alpha(n)} \frac{1}{|\bx|^{n-2}}, \quad &n\geq3.
    \end{cases}
  \end{equation}
  
\end{definition}

\medskip
\begin{theorem} (Representation formula)
  Let $\Omega = \calB^0(\bx, 1)$. 
  If $u \in \calC^2(\bar\Omega)$ is the solution of \eqref{eq:poisson}
  then the following holds
\begin{equation}\label{eq:sol}
  u(\bx) = -\int_{\p\Omega} u(\by) \frac{\p G}{\p\bnu}(\bx,\by) \ud S(\by)
  + \int_{\Omega} f(\by) G(\bx,\by) \ud \by \quad(\bx \in \Omega),
\end{equation}
where $\bnu$ is an outward normal vector, 
$\frac{\p G}{\p\bnu}(\bx,\by)$ is the normal derivative
of function $G(\bx, \by)$ at point $\by\in \p\Omega$.
\end{theorem}
\medskip
 
We denote by $\calD(\mR^n)$ the set of all functions $\phi:\mR^n\rightarrow \mR^n$ such that $\phi$ is infinitely differentiable
and has compact support. A function $u(\bx)$ is said to be a generalized or weak solution of \eqref{eq:poisson} if it solves the following integral equation for every $\phi \in \calD (\mR^n)$:
\begin{equation}\label{eq:poisson:weak}
  -\int_{\mR^n} u \Delta \phi \ud \bx = \int_{\mR^n} f \phi \ud \bx.
\end{equation}

Next we give the definition of bounded generalized functions:
\medskip
\begin{definition}[Bounded generalized function]
  We say that the distribution $g(\bx)$ is a bounded generalized 
  function in $\mR^n$, if for any function 
  $\varphi\in \calD(\mR^n)$, the function 
  $
  (g*\varphi)(\bx) = (g(\by), \varphi(\bx-\by)) 
  $
  is bounded in $\mR^n$, \ie
  $
  \sup|(g*\varphi)(\bx)| < \infty.
  $
\end{definition}
\medskip

It can be easily observed that the set of bounded generalized functions contains the set of usual bounded functions.

\section{Main results}\label{Sec:main_results}
In this section, we formulate the main results of this paper. The first result extends the results of Sibuya \cite{Sibuya_1971}.
We prove that under the same assumptions of Theorem~\ref{th:sibuya}, the function $u$ is continuously differentiable and
it's partial derivatives are almost periodic. 
\medskip
\begin{theorem}\label{th:1}
  Let $f(\bx)$ be an almost periodic function of $\bx$ in $\mR^n$, and let $u(\bx)$ be a bounded continuous function of $\bx$ in $\mR^n$. Assume that $u(\bx)$ is a solution of \eqref{eq:poisson} in the sense of distribution. Then $u$ has continuous partial derivatives $\p u/\p x_i$, which are almost periodic functions of $\bx$ in $\mR^n$.
\end{theorem}
\medskip

Theorem~\ref{th:1} generalizes the Sibuya's result in the case that not only $u$ is almost periodic, but also the partial derivatives $\p u/\p x_i$ are almost periodic. The second result of this work is to prove that $u$ does not have to
be a bounded continuous function in the usual sense. We prove that if the weak solution of the Poisson's equation is a bounded generalized function, then it is also a bounded continuous function in the usual sense.

\medskip
\begin{theorem}\label{th:2}
  Let $u$ be a bounded generalized function in $\mR^n$ which solves equation \eqref{eq:poisson} in the distribution sense. Then
  $u$ is a continuous and bounded function in $\mR^n$. 
  \end{theorem}
\medskip

In Section~\ref{sec:proof:th:1} we present details of the proof of Theorem~\ref{th:1}, and in Section~\ref{sec:proof:th:2} we give the proof of Theorem~\ref{th:2}.  Before proving our main theorems we need to prove several technical lemmas which are given in the next section.

\subsection{Technical lemmas}\label{Sec:tec_lemma}
First, we prove that if $u$ is a solution of the Poisson's equation in the distribution sense then it can be written in the form of the representation formula:

\medskip
\begin{lemma}\label{lemma:u:weak}
  Let $u(\bx)$ be a bounded continuous function of $\bx$ in $\mR^n$ that solves \eqref{eq:poisson} in the sense of distribution. Then, the representation formula \eqref{eq:sol} holds for $u$.
\end{lemma}
\begin{proof}
  We use the definition and properties of a standard mollifier. 
  Let $\omega$ be a standard mollifier defined as
  \begin{equation}
    \omega(\bx) = 
    \begin{cases}
      c\, \exp\Big( \frac{1}{|\bx|^2-1} \Big), \quad &|\bx| < 1,  \\
      0, \quad &|\bx| \geq 1,  
    \end{cases}
  \end{equation}
such that $\int_{|\bx| < 1} \omega(\bx) \ud \bx = 1$. Further, we define
\[
\omega_{\e}(\bx) = \frac{1}{\e^n} \omega\Big(\frac{\bx}{\e} \Big),
\]
and set $u_{\e} := \omega_{\e} * u$ and  $f_{\e} := \omega_{\e} * f$.

The idea of proof consist of two steps: first, we show that the following equality holds
\begin{equation}\label{eq:pe}
-\Delta u_{\e}(\bx) = f_{\e}(\bx).
\end{equation}
Then, we pass into the limit when $\e \rightarrow 0$, and use a uniform convergence property of the mollifier to get the 
desired result.

Using the definition of mollifier, we have
\begin{align*}
-\Delta u_{\e}(\bx) =& -\int_{\mR^n} \big( \Delta \omega_{\e}(\bx - \by) \big) u(\by) \ud \by, \\
f_{\e}(\bx) =& \int_{\mR^n} \omega_{\e}(\bx - \by) f(\by) \ud \by.
\end{align*}
Now, for a fixed $\bx$ we set $\varphi(\by) = \omega_{\e}(\bx - \by)$ and obtain
\[
-\int_{\mR^n} \Delta \varphi(\by) u(\by) \ud \by = \int_{\mR^n}  \varphi(\by) f(\by) \ud \by,
\]
which is in fact true, since $u$ solves the equation \eqref{eq:poisson} in the distribution sense. Therefore we conclude that \eqref{eq:pe} holds and since $u_{\e}$ is a smooth function the representation formula \eqref{eq:sol} can be written as
\begin{equation}\label{eq:sol:e}
  u_{\e}(\bx) = -\int_{\p\Omega} u_{\e}(\by) \frac{\p G}{\p\bnu}(\bx,\by) \ud S(\by)
  + \int_{\Omega} f_{\e}(\by) G(\bx,\by) \ud \by.
\end{equation}

Next, we write the following relation
\begin{align*}
  u_{\e}(\bx) - u(\bx) =& 
  \int_{\mR^n} \omega_{\e}(\bx - \by) u(\by) \ud \by - u(\bx) \\
  =& \int_{\mR^n} \frac{1}{\e^n}\omega\Big(\frac{\bx-\by}{\e} \Big) u(\by) \ud \by  - u(\bx).
\end{align*}
Let us denote $-\bxi = \frac{\bx - \by}{\e}$, then $\ud \by = \e^n \ud \bxi$ and
\begin{align*}
  u_{\e}(\bx) - u(\bx)  
  =& \int_{\mR^n} \omega(\bxi) u(\bx + \e\bxi) \ud \bxi  - u(\bx) \int_{\mR^n} \omega(\bxi) \ud \bxi \\
  =& \int_{\mR^n} \omega(\bxi) \big( u(\bx + \e\bxi) - u(\bx) \big) \ud \bxi.
\end{align*}
Consequently, we get that 
\begin{align*}
  |u_{\e}(\bx) - u(\bx)| \leq \sup_{\bx\in \mR^n, |\bxi| \leq 1} \big| u(\bx + \e\bxi) - u(\bx) \big|.  
\end{align*}
From here we conclude that if the right hand side  
goes to zero as  $\e \rightarrow 0$, then 
$u_{\e} \rightarrow u$ uniformly and analogously $f_{\e} \rightarrow f$. 
Now, passing to the limit in \eqref{eq:sol:e} as $\e \rightarrow 0$
we obtain the desired result.
\end{proof}

We next prove that a weak solution of the Poisson's equation in an
open unit ball has continuous partial derivatives in the ball.

\medskip
\begin{lemma}\label{lemma:u}
  Let a continuous in the unit ball $\calB(0,1) = \{\bx: |\bx| \leq 1 \}$
  function $u(\bx)$ be a solution of equation  \eqref{eq:poisson}
  in the distribution sense in $\calB^0(0,1) = \{\bx: |\bx| < 1 \}$.
  Then $u(\bx)$ 
  is continuously differentiable for all $\bx \in \calB^0(0,1)$. 
\end{lemma}

\begin{proof}
Using the formulation of solution \eqref{eq:sol} for $\bx \in \calB^0(0,1)$ and 
the definition of the Green's function \eqref{eq:G(x,y)} we get
\begin{equation}\label{eq:Gnu}
  \begin{aligned}
    \frac{\p G}{\p\bnu}(\bx,\by) =& 
    \sum_{i=1}^n y_i G_{y_i}(\bx, \by) \\ 
    = &
    -\frac{1}{n\alpha(n)}
    \frac{1}{|\bx-\by|^n}
    \sum_{i=1}^{n} y_i \big( (y_i - x_i) - y_i |\bx|^2 + x_i \big) \\
    = &
    -\frac{1}{n\alpha(n)}
    \frac{1-|\bx|^2}{|\bx-\by|^n}.
  \end{aligned}
\end{equation}

This implies that function $\frac{\p G}{\p\bnu}$ is continuously 
differentiable for  
$\bx \in \calB^0(0,1)$ and 
$\by \in \p\calB(0,1)$. 
Hence, the first term in the right hand side of \eqref{eq:sol}
is a differentiable function for $\bx \in \calB^0(0,1)$.

Analogously, from the definition of \eqref{eq:G(x,y)} it
follows that
\begin{equation}\label{eq:Gx}
  \begin{aligned}
    \frac{\p G}{\p x_i} =& \frac{\p\Phi}{\p x_i}(\by-\bx)
                            -\frac{\p\Phi}{\p x_i}(|\bx|(\by-\tilde\bx)) \\
    =&
       \frac{1}{n\alpha(n)} 
       \Big(
       \frac{ y_i-x_i}{|\by-\bx|^n} - 
       \frac{y_i - x_i |\by|^2}{| |\by| \bx - \frac{\by}{|\by|} |^n}
       \Big).
  \end{aligned}
\end{equation}
The first term of the right hand side of
$\frac{\p G}{\p x_i}$
has an integrable singularity when $\by = \bx$.
When $\by \not= 0$, the second term is continuous and bounded by
\[
\l| 
\frac{y_i - x_i |\by|^2}{| |\by| \bx - \frac{\by}{|\by|} |^n}
\r|
\leq \frac{2}{(1-|\bx|)^n},
\] 
which is integrable.
 
We conclude that the function $u(\bx)$ defined 
using the representation formula \eqref{eq:sol} 
in the unit ball $\Omega = \calB^0(0,1)$
\begin{equation}\label{eq:sol:B}
  u(\bx) = -\int_{\p\calB^0(0,1)} u(\by) \frac{\p G}{\p\bnu}(\bx,\by) \ud S(\by)
  + \int_{\calB^0(0,1)} f(\by) G(\bx,\by) \ud \by \quad(\bx \in \calB^0(0,1)),
\end{equation}
is continuously differentiable for $\bx \in \calB^0(0,1)$. 
\end{proof}

\medskip
\begin{remark}
  Note that Lemma \ref{lemma:u} is also true when the function $f\in L^{\infty}(\calB(0,1))$.
\end{remark}
\medskip

In the above lemma, we have proved that $u$ is differentiable if the Green's function and the normal derivative of the Green's function are differentiable with respect to $\bx$. The following lemma completes Lemma~\ref{lemma:u}:

\medskip
\begin{lemma}\label{lemma:lim}
  Let a continuous in the unit ball $\calB(0,1) = \{\bx: |\bx| \leq 1 \}$
  function $u(\bx)$ be a solution of equation  \eqref{eq:poisson}
  in the distribution sense in $\calB^0(0,1) = \{\bx: |\bx| < 1 \}$.
  Then, the following limits hold
  \begin{align}
    \lim_{h\rightarrow 0} \int_{\p\calB(0,1)} 
    \bigg| 
    \frac{\p}{\p x_i} \frac{\p G}{\p\bnu}(0, \by)
    - 
    \frac{\frac{\p G}{\p\bnu}(h\be_i, \by) - \frac{\p G}{\p\bnu}(0, \by)}{h}
    \bigg| \ud S(\by) = 0, \label{eq:lim:1} \\
    \lim_{h\rightarrow 0} \int_{\calB(0,1)} 
    \bigg| 
    \frac{\p G}{\p x_i}(0, \by)
    - 
    \frac{G(h\be_i, \by) - G(0, \by)}{h}
    \bigg| \ud \by = 0. \label{eq:lim:2}
  \end{align}
\end{lemma}

\begin{proof}
  By noting that $\by \in \p\calB(0,1)$ we have
  \begin{equation}
    \begin{aligned}
      \frac{\p}{\p x_i} \frac{\p G}{\p\bnu}(0,\by) 
      = &
      \frac{\p}{\p x_i} 
      \l(
      \frac{1}{n\alpha(n)}
      \frac{1-|\bx|^2}{|\bx-\by|^n}
      \r) \Big|_{\bx = 0} \\
      =& \frac{1}{n\alpha(n)}
      \l(
      \frac{2 x_i}{|\bx-\by|^n} - \frac{n(1-|\bx|^2)(x_i - y_i)}{|\bx-\by|^{n+2}}
      \r) \Big|_{\bx = 0} = \frac{y_i}{\alpha(n)}.
    \end{aligned}
  \end{equation}

  For the second term of \eqref{eq:lim:1} we get
  \begin{equation}
    \begin{aligned}
      \frac{1}{h}
      \l( 
      \frac{\p G}{\p\bnu}(h\be_i, \by)\r.& - \l. \frac{\p G}{\p\bnu}(0, \by)
      \r) \\ 
      =&
      \frac{1}{n\alpha(n)} \frac{1}{h} 
      \l(
      \frac{1-h^2}{|h\be_i -\by|^n} - 1
      \r) \\
      =&
      \frac{1}{n\alpha(n)} \frac{1}{h} 
      \l(
      - \frac{h^2}{|h\be_i -\by|^n} + \frac{1 - |h\be_i -\by|^n}{|h\be_i -\by|^n}
      \r).      
    \end{aligned}
  \end{equation}
  Now using the Taylor expansion 
  $
  |h\be_i - \by|^n = (|h\be_i - \by|^2)^{n/2} = (h^2 - 2h y_i + |\by|^2)^{n/2}
  = (1 - h(2 y_i + h))^{n/2} = 1 - \frac{n}{2} h (2 y_i + h) + \calO(h^2)
  = 1 - n h y_i + \calO(h^2)
  $, we obtain
  \begin{equation}
    \begin{aligned}
      \frac{1}{h}
      \l( 
      \frac{\p G}{\p\bnu}(h\be_i, \by)\r.& - \l. \frac{\p G}{\p\bnu}(0, \by)
      \r) \\ 
      =&
      \frac{1}{n\alpha(n)} \frac{1}{h} 
      \l(
       \frac{1 - \l(1 - n h y_i + \calO(h^2)\r)}{|h\be_i -\by|^n}
       - \frac{h^2}{|h\be_i -\by|^n}
      \r). \\      
      =&
      \frac{1}{n\alpha(n)} 
      \l(
       n  y_i +   \calO(h)
      \r)
      = 
      \frac{y_i}{\alpha(n)} + \calO(h).
    \end{aligned}
  \end{equation}
Consequently, the function inside the first integral of \eqref{eq:lim:1} converges uniformly 
to 0 with respect to $\by \in \p \calB(0,1)$.

Let us show that the second limit \eqref{eq:lim:2} is true. We start by splitting the 
unit ball $\calB(0,1)$ as the sum of two sets $\calB(0,\delta)$ and 
$\calB(0,1) - \calB(0,\delta)$, where $0<\delta<1$. Then for the first 
ball we have 
\begin{align*}
  \int_{\calB(0,\delta)} & 
  \l| 
  \frac{\p G}{\p x_i}(0, \by)
  - 
  \frac{G(h\be_i, \by) - G(0, \by)}{h}
  \r| \ud \by  \\
  \leq& \,
        \int_{\calB(0,\delta)} 
        \l| 
        \frac{\p G}{\p x_i}(0, \by)
        \r| \ud \by 
        + \frac1h 
        \int_{\calB(0,\delta)} 
        \l|
        G(h\be_i, \by) - G(0, \by)
        \r| \ud \by \\
  =:& \, I_1 + I_2.
\end{align*}
 
Using the formula \eqref{eq:Gx} at the point $(0, \by)$ we compute the integral $I_1$:
\begin{align*}
  I_1 = 
  \frac{1}{n\alpha(n)}\int_{\calB(0,\delta)} 
  \l| 
  \frac{y_i}{|\by|^n} - y_i
  \r| \ud \by 
  =& \,
  \frac{1}{n\alpha(n)}\int_{\calB(0,\delta)} 
  |y_i| 
  \l| 
  \frac{1- |\by|^n}{|\by|^n}
  \r| \ud \by \\
  \leq& \,
  \frac{1}{n\alpha(n)}\int_{\calB(0,\delta)} 
  \frac{1}{|\by|^{n-1}}
  \ud \by.
\end{align*}
We apply the definition of $G(\bx, \by)$ in \eqref{eq:G(x,y)} at the points
$(0, \by)$ and $(h\be_i, \by)$ and obtain
\[
I_2 \! = \!
\frac{1}{n(n-2)\alpha(n)}
\frac1h 
\int_{\calB(0,\delta)} \! 
\bigg| \!
\bigg(
\frac{1}{|h\be_i - \by|^{n-2}} - \frac{1}{|\by|^{n-2}}
\bigg)
+
\bigg(
1 - \frac{1}{\big| |\by| (h\be_i - \frac{\by}{|\by|^2}) \big|^{n-2}}
\bigg) \!
\bigg| 
\ud \by.
\]
Now we make use of the inequality of Lemma~A\ref{lemma:ab} to estimate expressions inside the brackets. By setting 
$a = 1/|h\be_i - \by|$ and $b = \frac{1}{|\by|}$ in inequality \eqref{eq:ab} and by using the fact that 
$|a| - |b| \leq |a-b|$ we can easily get 
\begin{align*}
  \bigg| 
  \frac{1}{|h\be_i - \by|^{n-2}} - \frac{1}{|\by|^{n-2}}
  \bigg|
  \leq& \,
        \frac{n-2}{2}
        \big|
        |\by| - |h\be_i - \by|
        \big|
        \bigg(
        \frac{1}{|h\be_i - \by|^{n-1}} + \frac{1}{|\by|^{n-1}}
        \bigg) \\
  \leq& \,
        \frac{n-2}{2}
        h 
        \bigg(
        \frac{1}{|h\be_i - \by|^{n-1}} + \frac{1}{|\by|^{n-1}}
        \bigg).
\end{align*}

Analogously, we set $a = 1$ and $b = 1/\big| |\by| h\be_i - \by/|\by| \big|$,
and by using the fact that
$
\big| |\by| h \be_i - \frac{\by}{|\by|} \big|
\geq \big| \frac{\by}{|\by|} - \big| |\by| h\be_i \big| 
\geq 1 - h
$ 
we get
\begin{align*}
  \bigg|
  1 - \frac{1}{\big| |\by| (h\be_i - \frac{\by}{|\by|^2}) \big|^{n-2}}
  \bigg|
  \leq& \,
        \frac{n-2}{2} 
        \bigg|
        1 - 
        \Big|
        |\by| h \be_i - \frac{\by}{|\by|}
        \Big|
        \bigg|
        \bigg(
        1 + \frac{1}{\big| |\by| h \be_i - \frac{\by}{|\by| }\big|^{n-1}}
        \bigg)
  \\
  \leq& \,
        \frac{n-2}{2}h \Big(1 + \frac{1}{(1-h)^{n-1}} \Big).
\end{align*}

Hence, we obtain the following estimate for the second integral $I_2$
\begin{equation}
  I_2 \leq 
  \frac{1}{2n\alpha(n)}
  \int_{\calB(0,\delta)} 
  \Bigg(
  \frac{1}{|h\be_i - \by|^{n-1}} + \frac{1}{|\by|^{n-1}}
  \Bigg) 
  \ud \by
  +
  \frac{\delta^n}{2n}
  \Big(1 + \frac{1}{(1-h)^{n-1}} \Big).
\end{equation}

From the above estimates for $I_1$ and $I_2$, we observe that the integrals contain two weakly singular integrals
at points $\by = 0$ and $\by = h\be_i$. By requiring that $h < \delta/2$ we guarantee that both points lie inside the ball $\calB(0, \delta)$. Using the standard techniques for computing weakly singular integrals we get 
\[
  \int_{\calB(0,\delta)} 
  \frac{1}{|\by|^{n-1}}
  \ud \by = 
  n\alpha(n) \delta,
\] 
and
\[
  \int_{\calB(0,\delta)} 
  \frac{1}{| h \be_i - \by|^{n-1}}
  \ud \by 
  < 
  \int_{\calB(h\be_i ,\delta + h)} 
  \frac{1}{| h \be_i - \by|^{n-1}}
  \ud \by 
  =
  n\alpha(n) (\delta + h).
\] 
Thus, we obtain the following estimate for the integral \eqref{eq:lim:2}
\begin{equation}
  I_1 + I_2 <
  \delta + \frac{\delta + h }{2} + \frac{\delta}{2} + 
  \frac{\delta^n}{2n}
  \Big(1 + \frac{1}{(1-h)^{n-1}} \Big).
\end{equation}

Eventually, in the second set $\calB(0,1) - \calB(0,\delta)$,
the function inside the integral \eqref{eq:lim:2}
is continuous, has no singular points and converges uniformly to
zero. Therefore, the integral of this function converges to zero.
 
The proof of the lemma is completed.
\end{proof}

\medskip
\begin{remark}
Let $\bx \in \calB^0(0,1)$ be a fixed point. Then, 
as proved in Lemma~\ref{lemma:lim}, one can prove that 
  \begin{align*}
    \lim_{h\rightarrow 0} \int_{\p\calB0,1)} 
    \bigg| 
    \frac{\p}{\p x_i} \frac{\p G}{\p\bnu}(\bx, \by)
    - 
    \frac{\frac{\p G}{\p\bnu}(\bx + h\be_i, \by) - \frac{\p G}{\p\bnu}(\bx, \by)}{h}
    \bigg| \ud S(\by) = 0, \\
    \lim_{h\rightarrow 0} \int_{\calB(0,1)} 
    \bigg| 
    \frac{\p G}{\p x_i}(\bx, \by)
    - 
    \frac{G(\bx + h\be_i, \by) - G(\bx, \by)}{h}
    \bigg| \ud \by = 0. 
  \end{align*}

\end{remark}

\section{Proof of the main theorems}\label{sec:proof:th}
\subsection{Proof of Theorem~\ref{th:1}}\label{sec:proof:th:1}
  Using the result of Theorem~\ref{th:sibuya}, we know that 
  $u$ is an almost periodic function. Lemma~\ref{lemma:u} 
  shows that the function $u(\bx_0+ \bx)$, where $\bx_0$ is a fixed
  point in $\mR^n$, is continuously differentiable, and moreover, 
  for its derivatives $\p u/ \p x_i$, $i=1, \ldots, n$, the following representation
  is true for any $\bx$ in  the unit ball $\calB^0(0,1)$:
  \begin{equation}\label{eq:th:du}
    \begin{aligned}
      \frac{\p u}{\p x_i}(\bx_0 + \bx) 
      =&\, 
      -\int_{\p\calB(0,1)} 
      u(\bx_0+\by) 
      \frac{\p}{\p x_i} \frac{\p G}{\p\bnu}(\bx,\by) \ud S(\by) \\
      &\, + \int_{\calB(0,1)} 
      f(\bx_0 + \by) \frac{\p G}{\p x_i}(\bx,\by) \ud \by.
    \end{aligned}
  \end{equation}

  The aim is to show  that the following relation holds:
  \begin{equation}\label{eq:th:lim}
    \lim_{h\rightarrow 0} \sup_{\bx_0 \in \mR^n} 
    \bigg|
    \frac{u(\bx_0+h\be_i) - u(\bx_0)}{h} - 
    \frac{\p u}{\p x_i}(\bx_0)
    \bigg|
    = 0.
  \end{equation}

  Let us start by writing the representation formula \eqref{eq:sol}
  at the point $(\bx_0 + \bx)$ for any $\bx \in \calB^0(0,1)$:
  \begin{equation}\label{eq:th:sol}
    u(\bx_0 + \bx) = -\int_{\p\calB(0,1)} u(\bx_0+\by) \frac{\p G}{\p\bnu}(\bx,\by) \ud S(\by)
    + \int_{\calB(0,1)} f(\bx_0 + \by) G(\bx,\by) \ud \by.
  \end{equation}
  
  From the relations \eqref{eq:th:du} and \eqref{eq:th:sol} it follows that
  \begin{align*}
    &\bigg|
      \frac{u(\bx_0+h\be_i) - u(\bx_0)}{h} - 
      \frac{\p u}{\p x_i}(\bx_0)
      \bigg| \\
    &\qquad \leq
           \sup_{\bx\in\calB(0,1)} |u(\bx)| 
           \int_{\p\calB(0,1)} 
           \bigg|
           \frac1h 
           \Big(
           \frac{\p G}{\p\bnu}(h\be_i,\by) - \frac{\p G}{\p\bnu}(0,\by)
           \Big) -
           \frac{\p}{\p x_i} \frac{\p G}{\p\bnu}(0,\by) 
           \bigg| 
           \ud S(\by) \\ 
    &\qquad +
        \sup_{\bx\in\calB(0,1)} |f(\bx)| 
        \int_{\calB(0,1)} 
        \bigg|
        \frac1h 
        \Big(
        G(h\be_i,\by) - G(0,\by)
        \Big) -
        \frac{\p G}{\p x_i}(0,\by) 
        \bigg| 
        \ud \by.      
  \end{align*}
  From the last inequality and Lemma~\ref{lemma:lim} the relation 
  \eqref{eq:th:lim} follows. 

  Now, according to Theorem~\ref{th:sibuya}, the function $u(\bx)$ is
  almost periodic and therefore 
  $
  \big(u(\bx_0+h\be_i) - u(\bx_0) \big)/h,
  $
  for $h>0,\, i = 1, \ldots, n$,
  is an almost periodic function. Hence, its limit $\p u/\p x_i(\bx_0)$,
  which is the uniformly continuous limit of this function, is also an almost 
  periodic function.  

  The proof of the theorem is completed.


\subsection{Proof of Theorem~\ref{th:2}}\label{sec:proof:th:2} The proof of the theorem is divided in three steps. First, by assuming $u$ is a continuous bounded function, we obtain the representation formula \eqref{eq:sol:B} for the ball of radius $r$ at the origin $\calB(0,r)$. Then, by the discussion of the proof of Lemma~\ref{lemma:u:weak}, we construct the representation formula for the generalized function $u$ and we prove that $u$ is continuous and bounded in the usual sense at the origin, \ie $\bx=0$. Last, we prove that $u$ is continuous and bounded for every point of $\mR^n$.  

1. Let us assume for the time being that $u$ is a bounded continuous function. 
Let us define the Green's function for a ball of radius $r$, \ie $\calB(0,r) \in \mR^n$:
\begin{equation}\label{eq:G_r}
  G_r(\bx,\by) := \Phi(\by - \bx) - \Phi\l(\frac{|\bx|}{r} (\by - \tilde\bx) \r),
  \quad \big( (\bx, \by) \in \calB(0,r), \bx \not= \by \big),
\end{equation}
where 
$
\tilde\bx = \frac{r^2\bx}{|\bx|^2}.
$
and $\Phi(\bx)$ is defined as in \eqref{eq:Phi(x,y)}.

Then, the representation formula in a ball of radius $r$ can be written as
\begin{equation}\label{eq:sol:r}
  \begin{aligned}
    u(\bx) 
    =&\, 
       -\int_{\p\calB(0,r)} u(\by) \frac{\p G_r}{\p\bnu}(\bx,\by) \ud S(\by)
       + \int_{\calB(0,r)} f(\by) G_r(\bx,\by) \ud \by \\
    =&\, 
       \frac{r^2 - |\bx|^2}{n\alpha(n) r} \int_{\p\calB(0,r)} \frac{u(\by)}{|\bx-\by|^n} \ud S(\by)
       + \int_{\calB(0,r)} f(\by) G_r(\bx,\by) \ud \by.
  \end{aligned}
\end{equation}
Now we multiply \eqref{eq:sol:r} by $\varphi(r)r^{n-1} \not=0$, set
$\bx=0$ and integrate for variable $r\in[0,R]$ for some $R>r$. By
noting that $|\by|=r$ we obtain:
\begin{equation}\label{eq:u0}
  \begin{aligned}
    u(0)  {\underbrace{ \int_0^R \varphi(r)r^{n-1} \ud r }_{I_1}}
    =&\,
    {\underbrace{ \int_0^R \frac{\varphi(r)r^{n}}{n\alpha(n)r^n}
       \int_{\p\calB(0,r)}  u(\by) \ud S(\by) \, \ud r  }_{I_2}}\\
       &\,  + 
       {\underbrace{ \int_0^R \varphi(r)r^{n-1} \int_{\calB(0,r)} f(\by) G_r(0,\by) \ud \by \, \ud r }_{I_3(f)}}.
  \end{aligned}
\end{equation}

Upon passing to the spherical coordinates $(r, \phi_1, \ldots, \phi_{n-1})$, where 
$r$ is the radial distance, $\phi_i$, $i=1,\ldots,n-1$ are angular 
coordinates, and by noting that 
$
  \ud S(\by) = \sin^{n-2}(\phi_1)  \sin^{n-3}(\phi_2)\cdots  \sin^{}(\phi_{n-2})   
  \ud \phi_1 \ud \phi_2 \cdots \ud \phi_{n-2}
$
and 
$
  \ud \by = r^{n-1} \ud S(\by) \ud r
$
we obtain that
\[
I_2 = \frac{1}{n\alpha(n)}\int_{|y| \leq R} \frac{\varphi(|\by|) u(\by)}{ |\by|^{n-1}} \ud \by.
\]
By dividing both parts of \eqref{eq:u0} by $I_1$ we obtain:
\begin{equation}\label{eq:u0y}
  \begin{aligned}
    u(0)  
    =&\,
    \frac{1}{n\alpha(n)}
    \frac{1}{I_1}
    \int_{|y| \leq R} \frac{\varphi(|\by|) u(\by)}{ |\by|^{n-1}} \ud \by
    + \frac{I_3(f)}{I_1}.
  \end{aligned}
\end{equation}
\\

2. Now, let us assume $u$ is a generalized function. 
By repeating the above steps with $u_\e=u*\omega_\e$ 
and it's corresponding source function 
$f_\e=f*\omega_\e$, 
and in addition using Lemma~\ref{lemma:u:weak} 
the following relation can be obtained:
\begin{equation}\label{eq:ue-ud}
  \begin{aligned}
    u_\e(0) 
    =&\,
    \frac{1}{n\alpha(n)}
    \frac{1}{I_1}
    \int_{|y| \leq R} \frac{\varphi(|\by|)}{ |\by|^{n-1}} u_\e(\by) \ud \by +
    \frac{1}{I_1} 
    I_3(f_\e).
  \end{aligned}
\end{equation}

In our construction, the function $\varphi$ is chosen such that
$\varphi(|\by|) = 0$ when $|\by| < \delta$ or 
$|\by|>R-\delta$, where $\delta>0$ is some 
small number.
Then, it is clear that $\psi(\by):=\varphi(|\by|)/|\by|^{n-1}$ is a basic function,
\ie it is an element of $\calD(\mR^n)$. 
Let us denote a function $v(\bx)$ as follows:
\[
v(\bx) := (u*\psi)(\bx) = 
\big( 
u(\by - \bx), \psi(\by) 
\big).   
\]
Now using Lemma~A\ref{lem:fubini} 
the integral at the right hand side of \eqref{eq:ue-ud}
can be simplified as follows:
\begin{equation}\label{eq:ue-ud:I}
  \begin{aligned}
    \int_{|\by| \leq R} &\psi(\by) u_\e(\by) \ud \by  
    =
    \int_{|\by| \leq R} \psi(\by)  
    (u*\omega_\e)(\by) \ud \by \\ 
    =&\,
    \int_{\mR^n} \omega_\e(\by) (u*\psi)(\by)\ud \by 
    =
    \int_{\mR^n} \omega_\e(\by) v(\by) \ud \by
    =
    v_\e(0),
  \end{aligned}
\end{equation}
where $v_\e = v*\omega_\e$.
Using the identity \eqref{eq:ue-ud:I} we rewrite the equation \eqref{eq:ue-ud}
in the following form:
\begin{equation}\label{eq:u0eps0}
  \begin{aligned}
    u_\e(0)
    =
    \frac{1}{n\alpha(n)}
    \frac{1}{I_1}
    v_\e(0)
    +
    \frac{1}{I_1} 
    I_3(f_\e).
  \end{aligned}
\end{equation}
From here it follows that the expression at the right hand side
has a limit as $\e \rightarrow 0$, since $v_\e(0)$ and $f_\e$ have a limit as 
$\e \rightarrow 0$.
Consequently, we conclude that 
there exists a limit $u_\e(0)$ as $\e\rightarrow 0$.

3. In the second part  of the proof, under the assumption of the 
theorem we showed that $u$ is continuous and bounded at point $\bx=0$.
Here we show that it is in fact true for every point of the space.

For a fixed point $\bx_0\in \mR^n$ let us denote $w(\bx):=u(\bx+\bx_0)$,
where $w(\bx)$ is a the solution of equation 
\[
-\Delta w(\bx) = f_{\bx_0}(\bx), \quad \bx\in \mR^n,
\]
in the distribution sense, where we denote $f_{\bx_0}(\bx):=f(\bx+\bx_0)$.
Then, we have that $w_\e(\bx) = u_\e(\bx+\bx_0)$ and 
\[
v(\bx-\bx_0) = (u*\psi)(\bx-\bx_0) = 
\big(
w(\by-\bx),\psi(\by)
\big)
=
(w*\psi)(\bx), 
\]
and the equation \eqref{eq:u0eps0} for $w$ becomes as
\begin{align*}
  w_\e(0)
  =
  \frac{1}{n\alpha(n)}
  \frac{1}{I_1}
  v_\e(-\bx_0)
  +
  \frac{1}{I_1} 
  I_3(f_{\bx_0, \e}),
\end{align*}
which is
\begin{align*}
  u_\e(\bx_0)
  =
  \frac{1}{n\alpha(n)}
  \frac{1}{I_1}
  v_\e(-\bx_0)
  +
  \frac{1}{I_1} 
  I_3(f_{\bx_0, \e}),
\end{align*}
Now, again passing to a limit as $\e\rightarrow 0$ we see that 
$v_\e(-\bx_0)\rightarrow v(-\bx_0)$, 
$f_{\bx_0, \e}\rightarrow f_{\bx_0}$,
and therefore the limit of $u_\e(\bx_0)$ exists 
as 
$\e\rightarrow 0$ for every $\bx_0\in \mR^n$.

The proof of the theorem is completed.

\appendix
\label{Sec:appendix}
  \section*{Appendix}
  \medskip
  \begin{applemma}\label{lemma:ab}
    Let $a$ and $b$ be arbitrary positive numbers. For any $m=1,2,3,\ldots$
    the following inequality holds
    \begin{equation}\label{eq:ab}
      |a^m - b^m| \leq \frac{m}{2} \bigg| \frac1b - \frac1a \bigg| (a^{m+1} + b^{m+1}).
    \end{equation}
  \end{applemma}

\begin{proof}
  Using the polynomial identities we have
  \begin{equation}\label{eq:ap:1}
  \begin{aligned}
    a^m - b^m =& (a-b) \sum_{k=1}^{m-1} a^{m-k} b^{k-1} \\
    =& \frac{a-b}{ab} \, ab\sum_{k=1}^{m-1} a^{m-k} b^{k-1} 
       = \Big( \frac1b -\frac1a \Big) \sum_{k=0}^{m-1} a^{m-k} b^{k+1}.
  \end{aligned}
\end{equation}

Next, we need to use the following H\"older's inequality 
\[
dc \leq \frac{d^p}{p} + \frac{c^q}{q}, 
\]
where $d>0, c>0, p>1, p>1 
\text{ and } \frac1p + \frac1q = 1$. By setting $d = a^{m-k}$,
$c = b^{k+1}$, $p=\frac{m+1}{m-k}$ and $q=\frac{m+1}{k+1}$ 
in the H\"older's inequality we get
\[
a^{m-k} b^{k+1} \leq \frac{m-k}{m+1} a^{m+1} + \frac{k+1}{m+1} b^{m+1}.
\]
Now, we apply this inequality for  \eqref{eq:ap:1} and obtain
\begin{align*}
|a^m - b^m| 
\leq& \Big| \frac1b - \frac1a \Big| 
\Big(a^{m+1} \sum_{k=0}^{m-1}\frac{m-k}{m+1} + b^{m+1} \sum_{k=0}^{m-1} \frac{k+1}{m+1} \Big) \\
=& \frac{m}{2} \Big| \frac1b - \frac1a \Big| (a^{m+1} + b^{m+1}),
\end{align*}
which completes the proof of the lemma.
\end{proof}

\medskip
\begin{applemma}\label{lem:fubini}
  Let $u$ be a distribution in $\mR^n$, 
  and let $\varphi(\bx)$ and $\psi(\bx)$ be two basic functions from 
  $\calD(\mR^n)$ such that 
  $
  \varphi(-\bx) = \varphi(\bx)
  $,
  and
  $
  \psi(-\bx) = \psi(\bx)
  $. Then the following equality holds
  \begin{equation}\label{eq:fubini}
    \int_{\mR^n} \varphi(\bx) \, (u * \psi)(\bx) \ud \bx
    =
    \int_{\mR^n} \psi(\bx) \, (u * \varphi)(\bx) \ud \bx
  \end{equation}
\end{applemma}

\begin{proof}
  Let us denote the support of functions $\psi$ and $\varphi$
  by $V\subset \mR^n$. Let $V_k$, $k=1,\ldots,N$,
  where $N$ is some finite number,
  be a set of disjoint simplexes such that 
  $\overline{V} = \cup_{k=1,\ldots,N} \overline{V}_k$,
  where $\overline{V}$ and $\overline{V}_k$ denote
  the closure of $V$ and $V_k$, respectively.
  Then for a fixed point $\bx_k \in V_k$, we write 
  the following equality:
  \begin{align*}
    \int_{\mR^n} \varphi(\bx) \, (u * \psi)(\bx) \ud \bx
      =&\, 
      \sum_{k=1}^N\int_{V_k} 
      \big(
      \varphi(\bx) \, (u * \psi)(\bx) 
      -
      \varphi(\bx_k) \, (u * \psi)(\bx_k)
      \big)
      \ud \bx \\
    &\,+
      \sum_{k=1}^N\int_{V_k} \varphi(\bx_k) \, (u * \psi)(\bx_k) \ud \bx.
  \end{align*}

  Since the integrand of the first integral of the right hand side is
  continuous, the first integral goes
  to zero when $k\rightarrow \infty$.  On the other hand, by noting
  that the integrand of the second integral is constant, and using the
  definition of distribution we obtain:
\begin{align*}
  \sum_{k=1}^N\int_{V_k} \varphi(\bx_k) \, (u * \psi)(\bx_k) \ud \bx
  =&\, 
     \Big(
     u(\by), \sum_{k=1}^N \varphi(\bx_k) \, \psi(\bx_k-\by) \,|V_k| 
     \Big),
\end{align*}
where $|V_k|$ denotes the volume $V_k$.

Now, when we pass to the limit when $k\rightarrow \infty$,
we get that
\begin{align}\label{eq:uphi}
  \int_{\mR^n} \varphi(\bx) \, (u * \psi)(\bx) \ud \bx
  =&\, 
     \Big(
     u(\by), \int_{\mR^n} \varphi(\bx) \, \psi(\bx-\by) \ud \bx
     \Big).
\end{align}

Analogously for the right hand side of \eqref{eq:fubini} we 
get
\begin{align}\label{eq:upsi}
  \int_{\mR^n} \psi(\bx) \, (u * \varphi)(\bx) \ud \bx
  =&\, 
     \Big(
     u(\by), \int_{\mR^n} \psi(\bx) \, \varphi(\bx-\by) \ud \bx
     \Big).
\end{align}
Finally, by using the assumption of the lemma, \ie 
$\varphi(-\bx) = \varphi(\bx)$
and 
$\psi(-\bx) = \psi(\bx)$,
we see that 
\begin{equation}
  \begin{aligned}
    \int_{\mR^n} \varphi(\bx) \, \psi(\bx-\by) \ud \bx
    =&\,
    \int_{\mR^n} \varphi(\bx) \, \psi(\by-\bx) \ud \bx \\
    =&\,
    \int_{\mR^n} \psi(\bx) \, \varphi(\by-\bx) \ud \bx 
    =
    \int_{\mR^n} \psi(\bx) \, \varphi(\bx-\by) \ud \bx,
  \end{aligned}
\end{equation}
where we used the property of mollifier. This completes the proof.
\end{proof}


\end{document}